\documentclass{amsart}
\usepackage{amsfonts}
\usepackage{amsmath}
\usepackage{amssymb}
\usepackage{amsthm}

\newtheorem{theorem}{Theorem}
\newtheorem{prop}{Proposition}
\newtheorem{lemma}{Lemma}

\newtheorem{exmp}{Example}
\newtheorem{cor}{Corollary}

\newtheorem*{prob}{Problem}

\begin{document}
\author{Mark Pankov}
\title{Automorphisms of infinite Johnson graph}
\address{Department of Mathematics and Informatics, University of Warmia and Mazury,
{\. Z}olnierska 14A, 10-561 Olsztyn, Poland}
\email{markpankov@gmail.com, pankov@matman.uwm.edu.pl}

\maketitle

\begin{abstract}
We consider the {\it infinite Johnson graph} $J_{\infty}$ whose vertex set consists of all subsets
$X\subset {\mathbb N}$
satisfying $|X|=|{\mathbb N}\setminus X|=\infty$ and whose edges are pairs of such subsets $X,Y$
satisfying $|X\setminus Y|=|Y\setminus X|=1$. An automorphism of $J_{\infty}$ is said to be
{\it regular}
if it is induced by a permutation on $\mathbb{N}$ or
it is the composition of the automorphism induced by a permutation on $\mathbb{N}$
and the automorphism $X\to {\mathbb N}\setminus X$. The graph $J_{\infty}$ admits non-regular automorphisms.
Our first result  states that
the restriction of every automorphism of $J_{\infty}$ to any
connected component ($J_{\infty}$ is not connected) coincides with the restriction of a regular automorphism.
The second result  is
a characterization of regular automorphisms of $J_{\infty}$ as
order preserving and order reversing bijective transformations of the vertex set of $J_{\infty}$
(the vertex set is partially ordered by the inclusion relation).
As an application, we describe automorphisms of the  associated {\it infinite Kneser graph}.
\end{abstract}

\section{Introduction}

\subsection{Classical Grassmann and Johnson graphs}
Let $V$ be an $n$-dimensional vector space (over a division ring)
and $n<\infty$.
The {\it Grassmann graph} $\Gamma_{k}(V)$ is the graph
whose vertex set is the Grassmannian ${\mathcal G}_{k}(V)$ formed by all $k$-dimensional subspaces of $V$
and whose edges are pairs of $k$-dimensional subspaces with $(k-1)$-dimensional intersections
(in what follows, two vertices of a graph joined by an edge will be called {\it adjacent}).
The graph $\Gamma_{k}(V)$ is connected.
By duality, $\Gamma_{k}(V)$ is isomorphic to $\Gamma_{n-k}(V^{*})$ ($V^{*}$ is the vector space dual to $V$).
Classical Chow's theorem \cite{Chow} states that every automorphism of $\Gamma_{k}(V)$, $1<k<n-1$,
is induced by a semilinear automorphism of $V$ or a semilinear isomorphism of $V$ to $V^{*}$;
the second possibility can be realized only in the case when $n=2k$.
If $k=1,n-1$ then any two distinct vertices of $\Gamma_{k}(V)$ are adjacent
and any bijective transformation of the vertex set is an automorphism of $\Gamma_{k}(V)$.
We refer \cite{Pankov1} for more information concerning Grassmann graphs.

The {\it Johnson graph} $J(n,k)$ is formed by all $k$-element subsets of $\{1,\dots,n\}$,
two such subsets are adjacent if their intersection consists of $k-1$ elements.
This graph admits a natural isometric embedding in $\Gamma_{k}(V)$.
Consider a base $B$ of the vector space $V$ and the subset of ${\mathcal G}_{k}(V)$
formed by all $k$-dimensional subspaces spanned by subsets of $B$.
Subsets of such type  are called {\it apartments} of ${\mathcal G}_{k}(V)$
(see \cite{Pankov1} for motivations of this term). Every apartment of ${\mathcal G}_{k}(V)$
is the image of an isometric embedding of $J(n,k)$ in $\Gamma_{k}(V)$.
However, the image of every isometric embedding of $J(n,k)$ in $\Gamma_{k}(V)$ is an apartment of ${\mathcal G}_{k}(V)$
only in the case when $n=2k$. This follows from the classification of isometric embeddings of
Johnson graphs $J(l,m)$ in $\Gamma_{k}(V)$ \cite{Pankov2}.
The graphs $J(n,k)$ and $J(n,n-k)$ are isomorphic:
the mapping $*$ transferring every subset $X\subset \{1,\dots,n\}$ to the complement
$\{1,\dots,n\}\setminus X$ defines an isomorphism between these graphs
(in the case when $n=2k$, this is an automorphism of $J(n,k)$).
It is not difficult to prove that every automorphism of $J(n,k)$
is induced by a permutation on $\{1,\dots,n\}$ or $n=2k$ and it is the composition of
the automorphism $*$
and the automorphism induced by a permutation on $\{1,\dots,n\}$
(an analog of Chow's theorem).

So, we can say that $J(n,k)$ is a "thin prototype" of $\Gamma_{k}(V)$.
Different characterizations of Grassmann and Johnson graphs can be found in \cite{Char1,Char2,Char3}, see also Sections 9.1 and 9.3 in \cite{BCN}.

\subsection{Grassmann graphs of infinite-dimensional vector spaces}
Now, suppose that $V$ is a vector space of infinite dimension $\aleph_{0}$.
Grassmannians of $V$ can be defined as the orbits of the action of the linear group ${\rm GL}(V)$ on
the set of all proper subspaces of $V$.
By \cite{Pankov1}, there are the following three types of Grassmannians:
\begin{enumerate}
\item[$\bullet$] ${\mathcal G}_{k}(V)$ formed by all subspaces of dimension $k\in {\mathbb N}$,
\item[$\bullet$] ${\mathcal G}^{k}(V)$ formed by all subspaces of codimension $k\in {\mathbb N}$,
\item[$\bullet$] ${\mathcal G}_{\infty}(V)$ formed by all subspaces of infinite dimension and codimension.
\end{enumerate}
Let ${\mathcal G}$ be one of these Grassmannians.
We say that $S,U\in {\mathcal G}$ are {\it adjacent} if
$$\dim(S/(S\cap U))=\dim(U/(S\cap U))=1.$$
The associated {\it Grassmann graph}, it will be denoted by
$\Gamma_{k}(V)$, $\Gamma^{k}(V)$ or $\Gamma_{\infty}(V)$ (respectively),
is the graph whose vertex set is ${\mathcal G}$ and whose edges are pairs of adjacent elements.

The graph $\Gamma_{k}(V)$ is connected and every automorphism of $\Gamma_{k}(V)$
is induced by a semilinear automorphism of $V$ \cite{Pankov1}.
By duality (see, for example, \cite{Baer,Pankov1}), $\Gamma^{k}(V)$ is canonically isomorphic to $\Gamma_{k}(V^{*})$.
Thus $\Gamma^{k}(V)$ is connected and every automorphism of $\Gamma^{k}(V)$
is induced by a semilinear automorphism of $V^{*}$.
The graph $\Gamma_{\infty}(V)$ is not connected. It admits automorphisms whose restrictions to distinct connected components
are induced by distinct semilinear isomorphisms \cite{BH}.
There is the following open problem \cite{Havlicek}.
\begin{prob}
Describe the restrictions of automorphisms of $\Gamma_{\infty}(V)$ to connected components.
\end{prob}
The idea used to prove Chow's theorem can not be exploited by many reasons, for example,
by the fact that the vector spaces $V$ and $V^{*}$ are non-isomorphic ($\dim V<\dim V^{*}$) \cite{Baer}.

\subsection{}
In this paper, a weak version of this problem will be solved.
We consider the {\it infinite Johnson graph} $J_{\infty}$ --- a thin prototype of $\Gamma_{\infty}(V)$.
The vertex set of $J_{\infty}$ is formed by all subsets $X\subset {\mathbb N}$
satisfying $|X|=|{\mathbb N}\setminus X|=\infty$, two such subsets $X,Y$ are adjacent if
$$|X\setminus Y|=|Y\setminus X|=1.$$
There is a natural isometric embedding of $J_{\infty}$ in $\Gamma_{\infty}(V)$:
for any infinite linearly independent subset $B\subset V$
consider the restriction of the graph $\Gamma_{\infty}(V)$ to the set formed by all elements of ${\mathcal G}_{\infty}(V)$
spanned by subsets of $B$.

An automorphism of $J_{\infty}$ will be called {\it regular} if
it is induced by a permutation on $\mathbb{N}$ or
it is the composition of the automorphism induced by a permutation on $\mathbb{N}$
and the automorphism $X\to {\mathbb N}\setminus X$.
The graph $J_{\infty}$ is not connected and admits non-regular automorphisms
(a simple modification of the example from \cite{BH}).
Our first result (Theorem 1) states that the restriction of every automorphism of $J_{\infty}$ to any
connected component of $J_{\infty}$ coincides with the restriction of a regular automorphism.
The vertex set of $J_{\infty}$ is partially ordered by the inclusion relation.
The second result (Theorem 2 and Corollary 1) is a characterization of regular automorphisms of $J_{\infty}$ as
order preserving and order reversing bijective transformations of the vertex set.
As an application of Theorem 2, we show that every automorphism of the associated {\it infinite Kneser graph} $K_{\infty}$
is induced by a permutation on $\mathbb{N}$.

Some general information concerning automorphisms of graphs can be found in \cite{Cam}.

\section{Infinite Johnson graphs}

\subsection{Definition}
Our definition of infinite Johnson graphs is similar to the definition of Grassmann graphs of
infinite-dimensional vector spaces given in the previous section.

Denote by ${\rm S}_{\infty}$ the group of all permutations on ${\mathbb N}$
and consider the action of this group on the set of all proper subsets of $\mathbb{N}$.
The associated orbits are of the following three types:
\begin{enumerate}
\item[(1)]
the set consisting of all $X\subset {\mathbb N}$ such that $|X|=k$
($k$ is a given natural number),
\item[(2)] the set consisting of all $X\subset {\mathbb N}$ such that $|{\mathbb N}\setminus X|=k$
(as in the previous case, $k$ is a given natural number),
\item[(3)] the set consisting of all $X\subset {\mathbb N}$ such that $|X|=|\mathbb{N}\setminus X|=\infty$.
\end{enumerate}
Let $J$ be one of these orbits.
We say that $X,Y\in J$ are {\it adjacent} if
$$|X\setminus Y|=|Y\setminus X|=1$$
(in the case (1), this condition is equivalent to the equality $|X\cap Y|=k-1$).
The associated {\it Johnson graph}, we will denote it by
$J_{k}$, $J^{k}$ or $J_{\infty}$ (respectively),
is the graph whose vertex set is $J$ and whose edges are pairs of adjacent elements.

\subsection{Some remarks on $J_{k}$ and $J^{k}$}
The mapping $*$ transferring every subset $X\subset {\mathbb N}$ to the complement ${\mathbb N}\setminus X$
defines an isomorphism between $J^{k}$ and $J_k$.
The structure of $J_{k}$ is rather similar to the structure of finite Johnson graphs.
This graph is connected. The distance between $X,Y\in J_k$ is equal to $|X\setminus Y|=|Y\setminus X|$
and the diameter of $J_k$ is $k$.
Maximal cliques of $J_k$ are the following two types:
\begin{enumerate}
\item[$\bullet$] the {\it star} $St(A)$, $A\in J_{k-1}$, consisting of all vertices of $J_{k}$ containing $A$,
\item[$\bullet$] the {\it top} $T(B)$,  $B\in J_{k+1}$, consisting of all vertices of $J_{k}$ contained in $B$.
\end{enumerate}
Every automorphism $f$ of $J_k$ preserves the class of maximal cliques (stars and tops).
Every top consists of precisely $k+1$ vertices and every star contains an infinite number of vertices;
this means that stars go to stars and tops go to tops.
In particular, $f$ induces a bijective transformation of the vertex set of $J_{k-1}$.
This transformation is an automorphism of $J_{k-1}$, since
two stars in $J_k$ have a non-zero intersection (consisting of precisely one vertex) if and only if
the associated vertices of $J_{k-1}$ are adjacent.
So, $f$ induces an automorphism of $J_{k-1}$. Step by step, we come to a permutation on $\mathbb{N}$
(an automorphism of $J_1$). This permutation induces $f$.
Now, suppose that $g$ is an automorphism of $J^k$. Then $h=*g*$ is an automorphism of $J_k$.
Hence $h$ is induced by a permutation $s\in {\rm S}_{\infty}$
and an easy verification shows that $g=*h*$ also is induced by $s$.
Therefore, {\it all automorphisms of the Johnson graphs $J_k$ and $J^{k}$ are induced by permutations on $\mathbb{N}$}.

\subsection{Basic properties of $J_{\infty}$}
The graph $J_{\infty}$ is not connected.
The connected component containing $X\in J_{\infty}$ will be denoted by $J(X)$;
it consists of all $Y\in J_{\infty}$ satisfying
$$|X\setminus Y|=|Y\setminus X|<\infty.$$
Any two connected components of $J_{\infty}$ are isomorphic
(every permutation on $\mathbb{N}$ induces an automorphism of $J_{\infty}$,
we consider a permutation transferring $X\in J_{\infty}$ to $Y\in J_{\infty}$,
the associated automorphism of $J_{\infty}$ sends $J(X)$ to $J(Y)$).

The graph $J_{\infty}$ contains an infinite number of connected components.
If $X\in J_{\infty}$ and $A$ is a finite subset of $X$ then $X\setminus A$ is a vertex of $J_{\infty}$
which does not belong to $J(X)$.
So, $X,Y\in J_{\infty}$ belong to distinct connected components if they are incident subsets of $\mathbb{N}$
($X\subset Y$ or $Y\subset X$).

Let $X\in J_{\infty}$. The {\it star} $St(X)$ consists of all $Y\in J_{\infty}$ containing $X$
and satisfying $|Y\setminus X|=1$.
Similarly, the {\it top} $T(X)$ is formed by all $Y\in J_{\infty}$ contained in $X$
and such that $|X\setminus Y|=1$.
Clearly, $St(X)$ and $T(X)$ both are maximal cliques of $J_{\infty}$
and it is easy to see that
{\it every maximal clique of $J_{\infty}$ is a star or a top}.

The automorphisms of $J_{\infty}$ induced by permutations on $\mathbb{N}$
map stars to stars and tops to tops.
The automorphism $*$ (sending every $X\in J_{\infty}$ to ${\mathbb N}\setminus X$)
transfers stars to tops and tops to stars.

\section{Automorphisms of $J_{\infty}$}

\subsection{Main results}
Recall that an automorphism of $J_{\infty}$ is {\it regular}
if it is induced by a permutation on $\mathbb{N}$ or it is the composition of the automorphism $*$
and the automorphism induced by a permutation on $\mathbb{N}$.
Note that $*f=f*$ for every automorphism $f$ of $J_{\infty}$ induced by a permutation on $\mathbb{N}$.

The vertex set of $J_{\infty}$ is partially ordered by the inclusion relation.
We say that a bijective transformation $f$ of the vertex set  is {\it order preserving}
or {\it order reversing} if it satisfies the condition
$$X\subset Y\;\Longleftrightarrow\; f(X)\subset f(Y)\;\;\;\;\;\forall\;X,Y\in J_{\infty}$$
or the condition
$$X\subset Y\;\Longleftrightarrow\; f(Y)\subset f(X)\;\;\;\;\;\forall\;X,Y\in J_{\infty},$$
respectively.
Every automorphism of $J_{\infty}$ induced by a permutation on $\mathbb{N}$ is order preserving.
The automorphism $*$ is order reversing. Therefore, every regular automorphism of $J_{\infty}$ is order preserving or order reversing;
in particular, all regular automorphisms of $J_{\infty}$ preserve the incidence relation.

Now we modify the example from \cite{BH} mentioned above
and establish the existence of non-regular automorphisms of $J_{\infty}$.

\begin{exmp}{\rm
Let $A\in J_{\infty}$ and $B$ be a vertex of the connected component $J(A)$.
We take any permutation $s\in {\rm S}_{\infty}$ sending $A$ to $B$.
This permutation preserves $J(A)$ and we define
$$f(X):=
\begin{cases}
s(X)&X\in J(A)\\
\;X&X\in J_{\infty}\setminus J(A).
\end{cases}$$
This is an automorphism of $J_{\infty}$.
We choose $Y\in J_{\infty}$ which is a proper subset of $A$ non-incident with $B$.
It is clear that $Y\not\in J(A)$, thus $f(Y)=Y$.
This means that $f$ does not preserve the incidence relation
($A$ and $Y$ are incident, but $f(A)=B$ and $f(Y)=Y$ are non-incident).
Therefore, the automorphism $f$ is non-regular.
}\end{exmp}

Our main result is the following.

\begin{theorem}\label{theorem1}
The restriction of every automorphism of $J_{\infty}$ to any connected component of $J_{\infty}$
coincides with the restriction of a regular automorphism to this connected component.
\end{theorem}

The second result is a characterization of regular automorphisms.

\begin{theorem}\label{theorem2}
Every order preserving bijective transformation of the vertex set of $J_{\infty}$
is the automorphism of $J_{\infty}$ induced by a permutation on $\mathbb{N}$.
\end{theorem}

Observe that for every order reversing bijective transformation $f$ of the vertex set of $J_{\infty}$
the mapping $*f$ is order preserving.
Thus, as a direct consequence of Theorem \ref{theorem2}, we get the following characterization of regular automorphisms of $J_{\infty}$.

\begin{cor}
The group of all regular automorphisms of $J_{\infty}$ coincides with
the group formed by all order preserving and order reversing bijective transformations of the vertex set of $J_{\infty}$.
\end{cor}

\subsection{Application: automorphisms of the infinite Kneser graph}
Recall that the {\it Kne\-ser graph} $K(n,k)$ and the Johnson graph $J(n,k)$ have the same vertex set; two vertices of $K(n,k)$ are adjacent if
they are disjoint subsets of $\{1,\dots,n\}$ (here we assume that $k<n-k$).
Every automorphism of $K(n,k)$ is induced by a permutation on $\{1,\dots,n\}$.
This follows from the Erd\H{o}s--Ko--Rado theorem;
see Section 7.8 in \cite{GR}.

Consider the {\it infinite Kneser graph} $K_{\infty}$ corresponding to the Johnson graph $J_{\infty}$.
The vertex set of this graph coincides with the vertex set of $J_{\infty}$ and
two vertices of $K_{\infty}$ are adjacent if they are disjoint subsets of $\mathbb{N}$.
This graph is a thin prototype of so-called {\it distant} graph defined for a vector space of dimension $\aleph_{0}$
\cite{BH}.
It is not difficult to prove that $K_{\infty}$ is a connected graph of diameter $3$.

\begin{cor}
Every automorphism of $K_{\infty}$ is induced by a permutation on $\mathbb{N}$.
\end{cor}

\begin{proof}
For every $X\in K_{\infty}$ denote by $X^{o}$ the set of all vertices of $K_{\infty}$ adjacent with $X$.
If $X,Y\in K_{\infty}$ then
$$X\subset Y\;\Longleftrightarrow\;Y^{o}\subset X^{o}.$$
This implies that every automorphism of $K_{\infty}$ is an order preserving transformation of the vertex
set of $K_{\infty}$. Since $K_{\infty}$ and $J_{\infty}$ have the same vertex set, Theorem 2 gives the claim.
\end{proof}

\section{Proof of Theorem \ref{theorem1}}
Let $A\in J_{\infty}$ and $f$ be the restriction of an automorphism of $J_{\infty}$ to
the connected component $J(A)$. Then $f(J(A))$ is a connected component of $J_{\infty}$.
It is clear that $f$ transfers maximal cliques of $J_{\infty}$ (stars and tops) contained in $J(A)$
to maximal cliques contained in $f(J(A))$.

\begin{lemma}\label{res1-lemma1}
One of the following possibilities is realized:
\begin{enumerate}
\item[{\rm (A)}] $f$ transfers stars to stars and tops to tops,
\item[{\rm (B)}] $f$ transfers stars to tops and tops to stars.
\end{enumerate}
\end{lemma}

\begin{proof}
We will use the following facts:
\begin{enumerate}
\item[$\bullet$] The intersection of two distinct stars $St(X)$ and $St(Y)$ is empty or contains
precisely one vertex; the second possibility is realized only in the case when $X,Y$ are adjacent vertices of $J_{\infty}$.
The same holds for the intersection of two distinct tops.
\item[$\bullet$] The intersection of a star $St(X)$ and a top $T(Y)$ is empty or consists of precisely two vertices;
the second possibility is realized only in the case when $X\subset Y$ and $|Y\setminus X|=2$.
\end{enumerate}
The proof is a direct verification.

Suppose that $J(A)$ contains a star $St(X)$, $X\in J_{\infty}$ such that $f(St(X))$ is a star.
Consider any $Y\in J_{\infty}$ adjacent with $X$. We choose $Z\in J_{\infty}$ satisfying
$$X\cup Y\subset Z\;\mbox{ and }\; |Z\setminus(X\cup Y)|=1.$$
Then
$$|St(X)\cap T(Z)|=|St(Y)\cap T(Z)|=2$$
and
$$|f(St(X))\cap f(T(Z))|=|f(St(Y))\cap f(T(Z))|=2.$$
Since $St(X)$ goes to a star, the latter equality guarantees that $f(T(Z))$ is a top
and $f(St(Y))$ is a star.

So, $f(St(Y))$ is a star for every $Y\in J_{\infty}$ adjacent with $X$.
Now consider an arbitrary  $Y\in J_{\infty}$ such that the star $St(Y)$ is contained in $J(A)$.
We take any
$$C_{0}\in St(X),\;\;C\in St(Y)$$
and consider a path
$$C_{0},\,C_{1},\,\dots,\,C_{i}=C$$
in $J(A)$ (a path joining $C_{0}$ and $C$ exists, since $J(A)$ is a connected component).
Then
$$X,\,C_{0}\cap C_{1},\,C_{1}\cap C_{2},\,\dots,\,C_{i-1}\cap C_{i},\,Y$$
is a path in $J_{\infty}$ (possible $X$ coincides with $C_{0}\cap C_{1}$ or $C_{i-1}\cap C_{i}$ coincides with $Y$).
It was established above that $St(C_{0}\cap C_{1})$ goes to a star.
Then, by the same arguments, the image of  $St(C_{1}\cap C_{2})$ is a star.
Step by step, we get that $f(St(Y))$ is a star.
Similarly, we establish that tops go to tops.

If $f$ transfers every star to a top then the same arguments show that tops go to stars.
\end{proof}

\begin{prop}\label{prop1}
In the case {\rm (A)}, $f$ is induced by a permutation on $\mathbb{N}$, i.e. there exists
$s\in {\rm S}_{\infty}$ such that
$$f(U)=s(U)\;\;\;\;\;\forall\;U\in J(A).$$
\end{prop}
Proposition \ref{prop1} will be proved in two steps --- Lemmas \ref{res1-lemma2} and \ref{res1-lemma3}.
In each of these lemmas, we assume that $f$ satisfies (A).

For every $X\in J_{\infty}$ we denote by $X^{\sim}$ the set consisting of $X$ and all vertices of $J_{\infty}$ adjacent with $X$.

\begin{lemma}\label{res1-lemma2}
For every $X\in J(A)$ the restriction of $f$ to $X^{\sim}$ is induced by a permutation on $\mathbb{N}$.
\end{lemma}

\begin{proof}
We can suppose that $f(X)=X$
(otherwise, we consider $tf$, where $t\in {\rm S}_{\infty}$ transfers $f(X)$ to $X$).
In this case, the restriction of $f$ to $X^{\sim}$ is a bijective transformation of $X^{\sim}$.

A star $St(U)$ is contained in $X^{\sim}$ if and only if
\begin{equation}\label{eq-res1-1}
U\subset X\;\mbox{ and }\; |X\setminus U|=1.
\end{equation}
Thus $f$ defines a permutation on the set of all $U$ satisfying \eqref{eq-res1-1}.
By Subsection 2.2, this permutation is induced by a certain permutation $s$ on $X$.

Now we extend $s:X\to X$ to a permutation on $\mathbb{N}$.
Let $n\in {\mathbb N}\setminus X$.
We choose $Y\in X^{\sim}$ containing $n$. Since $n\not\in X$, we have $X\ne Y$ and
$X,Y$ are adjacent. This means that $n$ is unique element of $Y\setminus X$.
Since $f(Y)$ and $f(X)=X$ are adjacent, $f(Y)\setminus X$ consists of precisely one element.
We denote this number by $s(n)$.

Show that our definition of $s(n)$ does not depend on $Y$.
Let us take any $Z\in X^{\sim}\setminus\{Y\}$ containing $n$.
Since $Y$ and $Z$ both are adjacent to $X$, we have
$$|X\setminus (X\cap Y)|=|X\setminus (X\cap Z)|=1.$$
If $X\cap Y$ coincides with $X\cap Z$ then $Y=Z$
(recall that $n$ belongs to both $Y,Z$ and $n\not\in X$).
Therefore, $X\cap Y$ and $X\cap Z$ are adjacent vertices of $J_{\infty}$.
The latter guarantees that $Y$ and $Z$ are adjacent.
Thus
\begin{equation}\label{eq-res1-2}
f(Y)=\{s(n)\}\cup (X\cap f(Y))\;\mbox{ and }\;f(Z)=\{n'\}\cup (X\cap f(Z))
\end{equation}
are adjacent (here $n'$ is unique element of $f(Z)\setminus X$).
Note that
\begin{equation}\label{eq-res1-3}
X\cap f(Y)\ne X\cap f(Z).
\end{equation}
Indeed, the equality
$$X\cap f(Y)=X\cap f(Z)$$ implies the existence of a star containing
$$f(X)=X,\,f(Y),\,f(Z);$$
however, there is no star containing $X,Y,Z$ (these vertices are contained in a top).
Since $f(Y)$ and $f(Z)$ are adjacent, \eqref{eq-res1-2} and \eqref{eq-res1-3} show that $s(n)=n'$.

It is clear that $s:{\mathbb N}\to {\mathbb N}$ is a permutation on ${\mathbb N}$
and $f(U)=s(U)$ for every $U\in X^{\sim}$.
\end{proof}

So, for every $X\in J(A)$ there is a permutation $s_{X}\in {\rm S}_{\infty}$ such that
$$f(U)=s_{X}(U)\;\;\;\;\;\forall\; U\in X^{\sim}.$$

\begin{lemma}\label{res1-lemma3}
If $X,Y\in J(A)$ are adjacent then $s_{X}=s_{Y}$.
\end{lemma}

\begin{proof}
Suppose that
$$X=\{n\}\cup(X\cap Y)\;\mbox{ and }\;Y=\{m\}\cup(X\cap Y).$$
We can assume that $f(X)=X$ and $f(Y)=Y$. Indeed, in the general case we have
$$f(X)=\{n'\}\cup(f(X)\cap f(Y)),\;\;f(Y)=\{m'\}\cup(f(X)\cap f(Y))$$
and consider $tf$, where $t\in {\rm S}_{\infty}$ transfers $n',m'$ and $f(X)\cap f(Y)$ to $n,m$ and $X\cap Y$, respectively.

It is easy to see that
$$X^{\sim}\cap Y^{\sim}=St(X\cap Y)\cup T(X\cup Y).$$
We have
$$s_{X}(X\cap Y)=s_{X}(X)\cap s_{X}(Y)=f(X)\cap f(Y)=X\cap Y.$$
Similarly, we get
$$s_{Y}(X\cap Y)=X\cap Y.$$
Then $s_{X}(n)=s_{Y}(n)=n$ and $s_{X}(m)=s_{Y}(m)=m$.

Let $k\in {\mathbb N}\setminus(X\cap Y)$. Then
$$U:=\{k\}\cup(X\cap Y)\in St(X\cap Y)\subset X^{\sim}\cap Y^{\sim}$$
and
$$s_{X}(U)=\{s_{X}(k)\}\cup (X\cap Y),\;\;s_{Y}(U)=\{s_{Y}(k)\}\cup (X\cap Y).$$
The equality $$s_{X}(U)=f(U)=s_{Y}(U)$$ shows that $s_{X}(k)=s_{Y}(k)$.

Let $k\in X\cap Y$. Then
$$W:=\{n,m\}\cup[(X\cap Y)\setminus\{k\}]\in T(X\cup Y)\subset X^{\sim}\cap Y^{\sim}$$
and
$$s_{X}(W)=\{n,m\}\cup[(X\cap Y)\setminus\{s_{X}(k\})],\;\;s_{Y}(W)=\{n,m\}\cup[(X\cap Y)\setminus\{s_{Y}(k\})].$$
The equality $$s_{X}(W)=f(W)=s_{Y}(W)$$ implies that $s_{X}(k)=s_{Y}(k)$.
\end{proof}

By connectedness, Lemma \ref{res1-lemma3} guarantees that $s_{X}=s_{Y}$ for all $X,Y\in J(A)$.
Proposition \ref{prop1} is proved.

In the case (B), the mapping $*f$ transfers stars to stars and tops to tops;
hence it is induced by a permutation on $\mathbb{N}$.
Thus $f$ is the composition of $*$ and the mapping induced by a permutation on $\mathbb{N}$.

\section{Proof of Theorem \ref{theorem2}}
Let $f$ be an order preserving bijective transformation of the vertex set of $J_{\infty}$.

\begin{lemma}\label{res2-lemma1}
Let $\{X_{i}\}_{i\in I}$ be a family of vertices of $J_{\infty}$ {\rm (}possible infinite{\rm )}
such that
$$X:=\bigcap_{i\in I}X_{i}\in J_{\infty}$$
Then
$$f(X)=\bigcap_{i\in I}f(X_{i}).$$

\end{lemma}

\begin{proof}
Since $f$ is order preserving, $f(X)$ is contained in every $f(X_{i})$ and we have
\begin{equation}\label{eq-res2-1}
f(X)\subset \bigcap_{i\in I}f(X_{i}).
\end{equation}
The inclusion
$$f(X)\subset \bigcap_{i\in I}f(X_{i}) \subset f(X_{i})$$
and the fact that $f(X),f(X_{i})$ are vertices of $J_{\infty}$ guarantee that
$$X':=\bigcap_{i\in I}f(X_{i})\in J_{\infty}.$$
The inverse mapping $f^{-1}$ is order preserving and
$f^{-1}(X')$ is contained in every $X_{i}$. Thus
$$f^{-1}(X')\subset X\;\mbox{ and }\;X'\subset f(X).$$
By \eqref{eq-res2-1}, $f(X)\subset X'$.
Therefore, $f(X)=X'$.
\end{proof}

\begin{lemma}\label{res2-lemma1'}
If $X,Y\in J_{\infty}$, $Y\subset X$ and $|X\setminus Y|=1$ then
$$|f(X)\setminus f(Y)|=1.$$
\end{lemma}

\begin{proof}
It is clear that $f(Y)$ is a proper subset of $f(X)$. If $|f(X)\setminus f(Y)|>1$ then there exists $Z\in J_{\infty}\setminus\{f(X),f(Y)\}$ such that
$$f(Y)\subset Z\subset f(X).$$
Then
$$Y\subset f^{-1}(Z)\subset X.$$
Since $|X\setminus Y|=1$, the latter inclusions mean that $f^{-1}(Z)$ coincides with $X$ or $Y$, a contradiction.
\end{proof}

\begin{lemma}\label{res2-lemma2}
For every $X\in J_{\infty}$ there exists a bijective mapping  $s:X\to f(X)$
such that
$f(Y)=s(Y)$ for every $Y\in J_{\infty}$ contained in $X$.
\end{lemma}

\begin{proof}
We restrict ourself to the case then $f(X)=X$
(in the general case, we consider the mapping $tf$ with $t\in {\rm S}_{\infty}$ transferring $f(X)$ to $X$).
Denote by ${\mathcal X}$ the set of all $Y\subset X$ satisfying
$|X\setminus Y|=1$.
All elements of ${\mathcal X}$ are vertices of $J_{\infty}$ and, by Lemma \ref{res2-lemma1'},
$f$ defines a permutation on ${\mathcal X}$.
By Subsection 2.2, this permutation is induced by a certain permutation $s$ on $X$, i.e.
$$f(Y)=s(Y)\;\;\;\;\;\forall\;Y\in {\mathcal X}.$$
Every $Y\in J_{\infty}$ contained in $X$
can be presented as the intersection of a family $\{Y_{i}\}_{i\in I}$ of elements from ${\mathcal X}$ (possible infinite).
Then
$$f(Y)=\bigcap_{i\in I}f(Y_{i})=\bigcap_{i\in I}s(Y_{i})=s(Y)$$
(the first equality follows from Lemma \ref{res2-lemma1}).
\end{proof}

For every $n\in {\mathbb N}$ denote by $[n]$ the set of all vertices of $J_{\infty}$ containing $n$.
\begin{lemma}\label{res2-lemma3}
For every $n\in {\mathbb N}$ there exists $s(n)\in {\mathbb N}$ such that
$$f([n])=[s(n)].$$
\end{lemma}

\begin{proof}
We take $Y_{1},Y_{2}\in J_{\infty}$ satisfying
$$Y_1\cup Y_{2}\in J_{\infty}\;\mbox{ and }\;Y_{1}\cap Y_{2}=\{n\}.$$
Let $X:=Y_{1}\cup Y_{2}$.
Lemma \ref{res2-lemma2} implies the existence of
a bijection $s:X\to f(X)$
such that
$f(Y)=s(Y)$ for every $Y\in J_{\infty}$ contained in $X$.
Then
$$f(Y_1)\cap f(Y_2)=s(Y_1)\cap s(Y_2)=s(Y_{1}\cap Y_{2})=\{s(n)\}.$$
We show that the number $s(n)$ is as required.

Let $Z\in [n]$. If $Z$ has an infinite intersection with $X$ then $Z\cap X$ is an element of $[n]$
contained in $X$ and
$$f(Z\cap X)=s(Z\cap X)$$
contains $s(n)$.
The inclusion
$$f(Z\cap X)\subset f(Z)$$
guarantees that $f(Z)\in [s(n)]$.

Suppose that $Z\cap X$ is finite.
In this case, we decompose $Z\setminus X$ in the disjoint union of two infinite subsets $A,B$
and define
$$T:=(Z\cap X)\cup A.$$
Then $T\in J_{\infty}$; moreover, $$n\in T\subset Z\;\mbox{ and }\; X':=X\cup T\in J_{\infty}.$$
By Lemma \ref{res2-lemma2}, there exists a bijection $s':X'\to f(X')$ such that
$f(Y)=s'(Y)$ for every $Y\in J_{\infty}$ contained in $X'$.
Since
$$Y_1\cap Y_2\cap T\ne \emptyset$$
(this intersection contains $n$),
we have
$$f(Y_1)\cap f(Y_2)\cap f(T)=s'(Y_1)\cap s'(Y_2)\cap s'(T)=s'(Y_1 \cap Y_2\cap T)\ne\emptyset.$$
On the other hand, $$f(Y_1)\cap f(Y_2)=\{s(n)\}.$$
Hence $f(T)$ contains $s(n)$
and the inclusion $f(T)\subset f(Z)$ implies that $f(Z)$ belongs to $[s(n)]$.

So, we obtain that $f([n])\subset [s(n)]$. Applying the same arguments to the transformation $f^{-1}$,
we get the inverse inclusion.
\end{proof}

The mapping $n\to s(n)$ is a permutation on ${\mathbb N}$
and $f$ is the automorphism of $J_{\infty}$ induced by
this permutation.

\end{document}